\documentclass[10pt,a4paper]{article}
\usepackage{biblatex} 
\usepackage[english]{babel}
\usepackage{amsmath}
\usepackage{amssymb}
\usepackage{amsthm}
\usepackage{authblk}
\usepackage{booktabs}
\usepackage{dcolumn}
\usepackage{fancyhdr}
\usepackage{geometry}
\usepackage{graphicx}
\usepackage{hyperref}
\hypersetup{colorlinks,allcolors=black}
\usepackage{latexsym}

\usepackage{graphicx}
\usepackage{siunitx}
\usepackage{tipa}
\usepackage{stackengine}
\usepackage{mathrsfs}
\usepackage{breqn}
\usepackage{dirtytalk}
\usepackage{pxfonts}

\newtheorem{theorem}{Theorem}[section]
\newtheorem{definition}{Definition}[section]

\theoremstyle{remark}
\newtheorem*{remark}{Remark}

\theoremstyle{example}
\newtheorem{example}{Example}

\theoremstyle{lemma}
\newtheorem{lemma}[theorem]{Lemma}

\theoremstyle{lemma}
\newtheorem{Corollary}[theorem]{Corollary}

\title{Long-Time Asymptotics of \\ the Sliced-Wasserstein Flow}

\author{Giacomo~Cozzi\thanks{\scriptsize Dipartimento di Matematica ``Tullio Levi-Civita'', Università degli Studi di Padova,  \\Via Trieste, 63, 35131 Padova PD, Italy. {\tt cozzi@math.unipd.it}},  Filippo~Santambrogio\thanks{\scriptsize Universite Claude Bernard Lyon 1, ICJ UMR5208, CNRS, Ecole Centrale de Lyon, INSA Lyon, Universit\'e Jean Monnet,\\ 43 boulevard du 11 novembre 1918,
69622 Villeurbanne, France. {\tt santambrogio@math.univ-lyon1.fr}}}

\begin{document}
\maketitle

\begin{abstract}
The sliced-Wasserstein flow is an evolution equation where a probability density evolves in time, advected by a velocity field computed as the average among directions in the unit sphere of the optimal transport displacements from its 1D projections to the projections of a fixed target measure. This flow happens to be the gradient flow in the usual Wasserstein space of the squared sliced-Wasserstein distance to the target. We consider the question whether in long-time the flow converges to the target (providing a positive result when the target is Gaussian) and the question of the long-time limit of the flow map obtained by following the trajectories of each particle. We prove that this limit is in general not the optimal transport map from the starting measure to the target. Both questions come from the folklore about sliced-Wasserstein and had never been properly treated.
\end{abstract}




\section{Introduction}
\label{S:0}
Solving problems related to optimal transport, such as computing optimal transport maps and Wasserstein distances, is in general very expensive from the numerical point of view, when we consider measures on $\mathbb{R}^d$ with $d\geq 2$.
The one-dimensional case, on the contrary, can be treated easily: explicit formulas for optimal transport maps, for example, are given by the cumulative distribution functions of the starting and target densities and their pseudo-inverses; also, the usually highly non linear Monge-Ampére equation, becomes an ODE in this case.  
Any kind of transport construction where the multi-dimensional computation of optimal transport maps is replaced by (many) one-dimensional computations are therefore of extreme interest. Among the first idea of this type we mention, for instance, the Knothe-Rosenblatt transport map (see \cite{Knothe57} or \cite[Section 2.3]{santambrogio2015optimal}).

A different construction was inspired by the so called Iterative Distribution Transfer (IDT) algorithm (first proposed in \cite{pitie2007automated}), which was based on a sequence of one-dimensional optimal matchings between projections of the measures along different axes, chosen randomly at each time step. In order to avoid the dependence of this approximation scheme on the choice of the axes, M. Bernot later proposed an homogenization of the IDT procedure. We will see later that this approach is quite related to the notion of sliced-Wasserstein distance, which has been the object of many works in the last decade, starting from \cite{rabin2012wasserstein}. Yet, not a lot can be found in the literature about this evolutionary procedure, which is now called sliced-Wasserstein flow and which we try to describe here, mainly based on \cite{Ber-per} and \cite[Section 2.5.2]{santambrogio2015optimal}.

Suppose that $\rho_0$ and $\nu$ are absolutely continuous probability measures in $\mathbb{R}^d$ (we will use the same notation to indicate both an absolutely continuous probability and its density with respect to the Lebesgue measure $\mathscr{L}^d$), both with finite second moment. 
For any unitary direction $\vartheta \in \mathbb{S}^{d-1}$ we can define $\rho_0^\vartheta=(\pi^\vartheta)_{\#}\rho_0$ and $\nu^\vartheta=(\pi^{\vartheta})_{\#}\nu$, where $\pi^{\vartheta} \colon \mathbb{R}^{d} \to \mathbb{R}$ is the canonical projection on the line $\vartheta \mathbb{R}$ and the symbol $\#$ indicates the push-forward of the measure. 
Since $\rho_0^\vartheta$ and $\nu^\vartheta$ are one-dimensional probability measures (absolutely continuous, hence atomless), computing the optimal transport map $T^{\vartheta}$ between them is not difficult. 
Then, for every point $x \in \text{spt}(\rho_0)$ and any direction $\vartheta$, we want to move the point $x$ in the direction $\vartheta$ of a displacement given by the displacement of the map $T^{\vartheta}$ at the point $x \cdot \vartheta$, and combine these displacements for all values of $\vartheta$. We then define a vector field $v$ via 
\begin{equation}\label{defivrhonu}
    v(x):=\fint_{\mathbb{S}^{d-1}} \left(T^\vartheta(x\cdot\vartheta)-x\cdot\vartheta\right)\vartheta d\mathscr{H}^{d-1}(\vartheta).
    \end{equation}
When needed, this vector field will be denoted by $v[\rho,\nu]$, as it depends on both the current measure $\rho$ and on the target one $\nu$, but when there is no ambiguity we will just use the notation $v$.

Fixing a small time step $\tau > 0 $, we can define $\rho^{\tau}_1 \coloneqq (id+\tau v)_{\#}\rho_0$ and, iterating the construction, we can obtain a sequence $\{\rho^{\tau}_n\}_{n \in \mathbb{N}}$. Letting $\tau \to 0^{+}$ we recover an (absolutely continuous with respect to the $2$--Wasserstein metric) curve of probabilities $(\rho_t)_{t \geq 0}$ solving the continuity equation
\begin{equation}
\label{CE}
    \partial_t \rho + \nabla \cdot (\rho v)=0,
\end{equation}
where $v=v[\rho,\nu]$ is obtained using the construction in \eqref{defivrhonu}.

The difference compared to the Iterative Distribution Transfer procedure is that in IDT we fix $\tau=1$ and we only use finitely many $\vartheta$, and more precisely we choose an orthonormal basis of $\mathbb R^d$. This continuous-in-time procedure is for sure smoother and more isotropic. 

It can be seen that the above equation is indeed a gradient flow for the Wasserstein distance $W_2$ (see \cite{ambrosio2005gradient,SantSurvey}) of the following functional
\[\rho \mapsto \frac{SW_2^2(\rho,\nu)}{2},\]
where $SW_2(\rho,\nu)$ is the \textit{2--sliced-Wasserstein distance} from $\rho$ to $\nu$, for which we refer to section \ref{S:2}.

From now on, we will call \textit{Sliced-Wasserstein Flow (SWF)} the curve $(\rho_t)_{t\geq 0}$ obtained above. If the distance $SW_2$ has received much attention in the recent applied literature (because of its better computability than the standard $W_2$ or because of its better statistical properties, see for instance \cite{rabin2012wasserstein,bonnotte2013unidimensional,bayraktar2021strong,Nadjahi2020StatisticalAT} as well as \cite[Section 5.5.4]{santambrogio2015optimal}), the flow itself has not. It was informally introduced by Bernot as a way to converge to the target measure, and few results are available on it: in \cite{bonnotte2013unidimensional} Bonnotte proved the existence of solutions to the SWF via JKO methods, and the same flow is also mentioned in \cite[Sections 2.5.2 and 8.4.2]{santambrogio2015optimal}, but uniqueness and questions related to the asymptotic behaviour of the flow still remain open. We cite \cite{SWFdiff} for a related flow, where diffusion is also added, and applications to generative flows. Yet, the presence of diffusion strongly changes the mathematical nature of these questions.

In this work we partially answer to two natural questions regarding the asymptotic behavior of the SWF:
\begin{itemize}
    \item \textbf{Question 1.} (Long-time asymptotics of the flow) Is it true that the limit $\rho_\infty=\lim_{t \to \infty} \rho_t$ (in any suitable weak sense) exists and that we have $\rho_\infty = \nu$?
    \item \textbf{Question 2.} (Optimality of the flow map) Fix an initial particle $x \in \mathbb{R}^d$ and consider the flow map describing the characteristics of the flow. This map is obtained by solving, for every inital point $x$, the following system
    \begin{equation}
    \label{ODE.}
        \begin{cases}
            \dot{y_x}=v(t,y_x),\\
            y_x(0)=x,
        \end{cases}
    \end{equation}
    where $v$ is the velocity field defined in \eqref{defivrhonu}. Assume that $y_x(t)$, solution of the system above, exists and is well defined for any $t\geq 0$ and for every $x \in \mathbb{R}^d$, we will denote by $Y_t$ the map $x\mapsto y_x(t)$. Moreover, assume that the limit $Y_{\infty}\coloneqq\lim_{t \to \infty} Y(t)$ is well defined and that $\lim_{t \to \infty} \rho_t=\nu$. Is it true that the map $Y_{\infty}$ is optimal in the sense of optimal transport between the initial datum $\rho_0$ and the target measure $\nu$?
\end{itemize}

A similar question concerning the optimality of the flow map was already studied in the case of another flow map, the one obtained for the Fokker Planck equation. Negative results were produced first in \cite{tanana2021comparison}, and then in \cite{lavenant2022flow}. The latter result also included the case where the target distribution is the standard Gaussian. Similarly to the problem studied in \cite{tanana2021comparison,lavenant2022flow}, the flow map of the sliced-Wasserstein flow has also been the object of numerical investigation: in the case of discrete measures \cite{pitie2007automated} suggests that the map $Y_\infty$ obtained as in Conjecture $2$ is a good approximation of the optimal one, but does not coincide with it in general. A similar result was also conjectured by Bernot \cite{Ber-per}.

This work is structured as follows: in the next section we recall some basic facts about optimal transport and Wasserstein spaces, then we present the definition and first properties of the sliced-Wasserstein distance and of the sliced-Wasserstein flow. In section \ref{S:3} we prove estimates for the long-time behavior of the flow, and we discuss Conjecture 1, furnishing examples that show the flow may or may not converge, depending on the starting and target measures chosen: in particular, we prove convergence of the flow provided the target measure is the Gaussian distribution. This positive result, together with the nature of the negative ones suggest that one could have convergence to the target as soon as the initial measure is of finite entropy, but this general result is unfortunately unproven. Finally, section \ref{S:4} contains the counterexample which shows that Conjecture 2 is false.

\section{Notations and some results in optimal transport}
\label{S:1}
For a function $u : \mathbb{R}^d \mapsto \mathbb{R}$, we denote by $\nabla u$ its
gradient and $D^2 u$ its Hessian matrix. Moreover, for a map $T : \mathbb{R}^d \mapsto \mathbb{R}^d$, its Jacobian matrix is denoted by $DT$. 
Given any open subset $\Omega \subset \mathbb{R}^d$, we denote with $\mathscr{P}(\Omega)$ the set of probability measures with support in $\Omega$. We denote with $\mathscr{P}_p(\Omega)$ the set of probability measures on $\Omega$ with finite moment of order $p$ (in case $\Omega$ is unbounded), so that we have $\mathscr{P}_p(\Omega)=\{\mu\in \mathscr{P}(\Omega)\,:\, \int|x|^pd\mu<+\infty\}$.

Given $\mu, \nu \in \mathscr{P}(\Omega)$, we say that $T$ is a transport map from $\mu$ to $\nu$, and that $\nu$ is the pushforward of $\mu$ via $T$ (and we write $\nu=T_{\#}\mu$) if for every Borel set $A \subset \Omega$ we have
\[\nu(A)=\mu(T^{-1}(A))\]
or, equivalently, if for any Borel map $\phi \colon \Omega \to \mathbb{R}$ it holds
\[\int_{\Omega} \phi(x) d\nu(x) =\int_{\Omega}\phi(T(x))d\mu(x).\]
\begin{definition}
    Let $\mu, \nu \in \mathscr{P}\left(\Omega\right)$ and a continuous cost function $c:\Omega\times\Omega\to\mathbb R$. 
    The map $T : \mathbb{R}^d \to \mathbb{R}^d$ is said to be an optimal transport map for the cost $c$ if it solves the problem
    \begin{equation}
        \label{Monge_pb}
        \min_{T} \left\{\int_{\Omega}c(x,T(x))d\mu(x), \hspace{2mm} T \hspace{2mm} \text{Borel map such that} \hspace{2mm} T_{\#}\mu=\nu \right\}.
    \end{equation}
    The map $\varphi: \mathbb{R}^d \to \mathbb{R}$ is said to be a Kantorovich potential if it solves the problem
    \begin{equation}
    \label{Potential_pb}
        \max_{\varphi}\left\{\int_{\Omega}\varphi(x)d\mu(x)+\int_{\Omega}\varphi^c(x)d\nu(x)\right\},
    \end{equation}
    where $\varphi^c(x) \coloneqq \inf_{y \in \mathbb{R}^d} \left\{c(x,y)-\varphi(y)\right\}$.
\end{definition}
\begin{remark} Problems \eqref{Monge_pb} and \eqref{Potential_pb} do not always admit a solution. However, when they exist, the two quantities defined by \eqref{Monge_pb} with \say{$\min$} replaced by \say{$\inf$}, and \eqref{Potential_pb} with \say{$\max$} replaced by \say{$\sup$} coincide. 

We are in particular interested in the case $\Omega=\mathbb R^d$ and $c(x,y)=\frac 12|x-y|^2$. In this case we have a relation between the optimal $T$ and the optimal $\varphi$ (which is a Lipschitz function, and hence diffentiable a.e.) which is
\begin{equation}
\label{KantVsTransp}
T(x)=x-\nabla \varphi(x).
\end{equation}
In particular, we have $T=\nabla u$ where $u(x)=\frac12|x|^2-\varphi(x)$, and one can check that $u$ is convex. This is summarized in the following celebrated theorem.
\end{remark}
\begin{theorem}[Brenier's theorem, see \cite{brenier1987decomposition}]
\label{Brenier's_thm}
    Let $\mu$, $\nu$ be two absolutely continuous probability measures in $\mathscr{P}_2(\mathbb R^d)$. A map $T \colon \mathbb{R}^d \to \mathbb{R}^d$ such that $T_{\#}\mu=\nu$ is optimal for the optimal transport problem \eqref{Monge_pb} with the cost $c(x,y)=\frac 12|x-y|^2$ if and only if there exists a convex function $u \colon \mathbb{R}^d \to \mathbb{R}$ such that $T=\nabla u$ a.e. This is equivalent to $DT$ being a symmetric positive semi-definite matrix valued measure on $\mathbb{R}^d$. Moreover if it is the case, then $u$ must satisfy the Monge-Ampére equation
    \begin{equation*}
        \label{Monge_Ampére}
        \det D^2 u(x)=\frac{\nu(x)}{\mu(T(x))}
    \end{equation*}
    everywhere on $\mathbb{R}^d$. Finally, optimal transport maps are stable by \say{inversion}: $T \colon \mathbb{R}^d \to \mathbb{R}^d$ is the optimal transport map from $\mu$ to $\nu$ if and only if $T^{-1}$ is the optimal transport map from $\nu$ to $\mu$.
\end{theorem}
For the one-dimensional problem, we are able to characterize optimal transport maps:
\begin{theorem}[Brenier's theorem in $d=1$, see \cite{santambrogio2015optimal}, Chapter 2]
    Let $\mu$, $\nu$ be two probability densities with finite second moments. If $\mu$ is atomless, then there exist a unique non-decreasing transport map from $\mu$ to $\nu$, which is optimal and given by
    \[T(x)=F_\nu^{[-1]}\left(F_{\mu}(x)\right),\]
    where $F_{\mu}(x)=\mu((-\infty,x])$ and $F^{[-1]}(y)=\inf\{t\in \mathbb{R} \colon F(t)\geq y\}$.
\end{theorem}


\section{Sliced-Wasserstein distance and sliced-Wasserstein flow}
\label{S:2}
\subsection{Wasserstein spaces}
Let $\Omega \subset \mathbb{R}^d$. Thanks to the transport value associated with the costs of the form $c(x,y)=|x-y|^p$ for $1 \leq p < \infty$, we can define a distance called \textit{p--Wasserstein distance} over $\mathscr{P}_p(\Omega)$, the space of probability measures with support contained in $\Omega$ and finite $p$-moment.

Wasserstein distances play a key role in many fields of applications, and seem to be a natural way to describe distances between equal amounts of mass distributed on the same space. 
\begin{definition}[Wasserstein distance]
    Let $1 \leq p < \infty$, $\Omega \subset \mathbb{R}^d$. Given two absolutely ontinuous measures $\mu, \nu \in \mathscr{P}_{p}(\Omega)$ we define 
    \begin{equation*}
        \label{Wasserstein_Distance}
            W_{p}(\mu, \nu)= \inf \left\{\int_{\Omega} |x-T(x)|^p d\mu(x) \hspace{1mm}\colon \hspace{2mm} T \hspace{2mm} \text{Borel map such that} \hspace{2mm} T_{\#}\mu=\nu\right\}^{\frac{1}{p}}.
    \end{equation*}
\end{definition}
The above definition could actually be extended to the whole set of measures in $\mathscr{P}_{p}(\Omega)$ (without the restriction to be absolutely continuous) but this requires to re-formulate the transport problem in terms of transport plans instead of transport maps. However, in this article this aspect will not be crucial and the only cases where non-absolutely continuous measures will be used concern some counter-examples where it is easy to understand the adaptations that need to be performed.

As a refence on the distance $W_p$ we refer the reader to \cite[Chapter 5]{santambrogio2015optimal} where it is proven, in particulat, that $W_p$ is indeed a distance over $\mathscr{P}_{p}(\Omega)$.  Moreover, the following theorem holds:

\begin{theorem}
    \label{Compactness_Wasserstein}
    Fix $p \geq 1$, $q> p$ and $c > 0$. Then the set \[ K \coloneqq \left\{\mu \in \mathscr{P}_{p}\left(\mathbb{R}^d\right) \colon \int_{\mathbb{R}^d}|x|^q d\mu(x) \leq c\right\}\] is sequentially compact in the metric space $\left( \mathscr{P}_p \left(\mathbb{R}^d\right), W_{p}\right)$.
\end{theorem}
\subsection{Sliced-Wasserstein}
The following definition was introduced in \cite{rabin2012wasserstein}, and provides a notion of distance on $\mathscr{P}_2(\mathbb{R}^d)$, alternative to the usual $W_2$ distance, based on the behavior of the measures \say{direction by direction}.
\begin{definition}[Sliced-Wasserstein distance]
    Given $\mu,\nu \in \mathscr{P}_2(\mathbb{R}^d)$, we define
    \begin{equation*}
    \label{SW_dist}
            SW_2 \coloneqq \left(\fint_{\mathbb{S}^{d-1}}W_2^2((\pi^{\vartheta})_{\#}\mu,(\pi^{\vartheta})_{\#}\nu)d\mathscr{H}^{d-1}(\vartheta)\right)^{1/2},
    \end{equation*}
    where $\pi^\vartheta \colon \mathbb{R}^d \to \mathbb{R}$ is the projection on the axis directed according to the unit vector $\vartheta$, namely $\pi^{\vartheta}(x)=x\cdot \vartheta$, and $\mathscr{H}^{d-1}$ is the surface measure on $\mathbb{S}^{d-1}$.
\end{definition}
It is not difficult to prove that, for any $\Omega \subset \mathbb{R}^d$,
this is a distance on $\mathscr{P}_2(\Omega)$, 
that is always less than the usual 2-Wasserstein distance.
The two distances generate the same topology (see \cite{bayraktar2021strong}), with Hölder relation on compact sets (see \cite{bonnotte2013unidimensional},
Chapter 5). Many interesting properties, pointing out surprising difference in the behavior of $SW_2$ compared to $W_2$, have been studied in the recent paper \cite{ParSle}. We also mention the possibility to define a distance $SW_p$, via $SW_p :=\left(\fint_{\mathbb{S}^{d-1}}W_p^p((\pi^{\vartheta})_{\#}\mu,(\pi^{\vartheta})_{\#}\nu)d\mathscr{H}^{d-1}(\vartheta)\right)^{1/p}$, which is sometimes more convenient for sharp comparisons with the standard Wasserstein distances (see \cite{CarFigMerWan}).

To show that equation \eqref{CE} derived in the introduction, describing the one-dimensional approximation scheme for a transport map, can be seen as the gradient flow of the functional $\mathscr{F}\colon \rho \mapsto \frac12 SW_2^2(\rho,\nu)$ in the Wasserstein space $(\mathscr{P}(\mathbb{R}^d), W_2)$ we need to convince ourselves that $v_t = -\nabla \frac{\delta \mathscr{F}}{\delta \rho}(\rho_t)$, 
where $\frac{\delta \mathscr{F}}{\delta \rho}(\rho)$  is the first variation of $\mathscr{F}$, defined by
\[\frac{d}{d\varepsilon}\bigg|_{\varepsilon=0} \mathscr{F}(\rho + \varepsilon \chi)=\int \frac{\delta \mathscr{F}}{\delta \rho}(\rho) d\chi\]
for any perturbation $\chi=\mu-\rho$ with $\mu \in L^{\infty}(\mathbb{R}^d)\cap \mathscr{P}(\mathbb{R}^d)$ and compactly supported. Since the first variation of the functional $\rho \mapsto \frac12 W_2^2(\rho,\nu)$ is the Kantorovich potential for the transport problem from $\mu$ to $\nu$ (see \cite{santambrogio2015optimal}, proposition 7.17), we get
\begin{equation*}
    \frac{\delta \mathscr{F}}{\delta \rho}(\rho)(x)=\fint_{\mathbb{S}^{d-1}} \frac{1}{2} \frac{\delta W_2^2(\rho^\vartheta,\nu^\vartheta)}{\delta \rho}(x) d\mathscr{H}^{d-1}(\vartheta)=\fint_{\mathbb{S}^{d-1}} \varphi^{\vartheta}(x\cdot \vartheta)d\mathscr{H}^{d-1}(\vartheta),
\end{equation*}
where we denoted by $\varphi^{\vartheta}$ the Kantorovich 
potential between $\rho^\vartheta \coloneqq (\pi^\vartheta)_{\#}\rho$ and
$\nu^\vartheta\coloneqq (\pi^\vartheta)_{\#}\nu$.  
Using also \eqref{KantVsTransp}, this in turns leads to
\begin{equation*}
    -\nabla \frac{\delta \mathscr{F}}{\delta \rho}(\rho)=\fint_{\mathbb{S}^{d-1}} \left(T^\vartheta_{t}(x\cdot \vartheta)-x\cdot\vartheta\right)\vartheta d\mathscr{H}^{d-1}(\vartheta),
\end{equation*}
which is precisely the velocity field derived in \eqref{defivrhonu}.

\section{Long-time asymptotics of the SWF}
\label{S:3}

The question of the limit as $t\to\infty$ of the distribution $\rho_t$ is very natural but surprisingly difficult, even if the SWF was probably mainly introduced as a way to approach the target distribution $\nu$. Section 5.7 in \cite{bonnotte2013unidimensional} explains the difficulty of this question as it is not clear, except in the case of strictly positive densities, that critical points should coincide with the target. Note that \cite{LiMoos} proves convergence results for a discrete (in time) version of a similar procedure, but among its assumption there is a very restrictive one: the existence of a compact set in the space of measures where the evolution takes place and which contains no other critical point than $\nu$. 
\subsection{Bounds on the moments along the flow}
In this section we provide some uniform bounds to the $p$-moments of the sliced-Wasserstein flow. We start by computing the derivative with respect to time of the $p$--moment ($p\geq 2$) of the sliced-Wasserstein flow $\rho_t$ associated with a given target measure $\nu$. We distinguish the case $p=2$ and $p>2$

For $p=2$, using the equality 
\begin{equation*}
    x\cdot \vartheta\left(T^\vartheta(x\cdot\vartheta)-x\cdot\vartheta\right)=-\frac{|T^\vartheta(x\cdot \vartheta)-x\cdot \vartheta|^2}{2}-\frac{|x\cdot\vartheta|^2}{2}+\frac{|T^\vartheta(x\cdot\vartheta)|^2}{2},
\end{equation*}
we have
\begin{equation}
\label{due}
    \begin{split}
        \partial_t \int \frac{|x|^2}{2} \rho_t &=\partial_{t} \int \frac{|x|^{2}}{2}\rho_t = -\int \frac{|x|^{2}}{2}\nabla\cdot(v \rho_t) = \int x\cdot v \rho_t \\
        &= \fint \int x \cdot \vartheta \left(T^\vartheta(x\cdot \vartheta)-x\cdot \vartheta\right) \rho_t = \fint \int_{\mathbb{R}} \frac{|T^\vartheta_t(y)|^2}{2}\rho_t^\vartheta(y)\\
        &-\fint \int_{\mathbb{R}} \frac{r^2}{2}\rho_t^\vartheta(y)-\frac{1}{2}\fint \int_{\mathbb{R}} |T^\vartheta_t(y)-y|^2 \rho_t^\vartheta(y) \\
        &=\fint \int_{\mathbb{R}} \frac{|T^\vartheta_t(y)|^2}{2}\rho_t^\vartheta(y)-\fint \int_{\mathbb{R}} \frac{r^2}{2}\rho_t^\vartheta(y)-\frac{1}{2} SW_{2}^{2}(\rho_t,\nu) \\
        &=\fint \int_{\mathbb{R}} \frac{r^2}{2}\nu^\vartheta(r)-\fint \int_{\mathbb{R}} \frac{r^2}{2}\rho_t^\vartheta(y)-\frac{1}{2} SW_{2}^{2}(\rho_t,\nu)
    \end{split}
\end{equation}

For $p>2$, on the other hand, we have (defining $\mathcal{M}_p(\rho)\coloneqq \int|x|^p \rho$)
\begin{equation*}
\begin{split}
\label{pmomenti}
    &\partial_{t}\mathcal{M}_{p}(\rho_t)= \partial_{t} \int \frac{|x|^{p}}{p}\rho_t = -\int \frac{|x|^{p}}{p}\nabla\cdot(v \rho_t) = \int |x|^{p-2}x\cdot v \rho_t \\
    &= \int |x|^{p-2}x \cdot \left(\fint (T^\vartheta_t(x\cdot \vartheta)-x\cdot \vartheta)\vartheta \right) \rho_t =  \fint \int |x|^{p-2}x\cdot \vartheta (T^\vartheta_t(x\cdot \vartheta)-x\cdot \vartheta)\rho_t \\
    &\leq  - \fint \int |x|^{p-2}\frac{|x\cdot \vartheta|^2}{2}\rho_t + \fint \int |x|^{p-2}\frac{|T^\vartheta_t(x\cdot \vartheta)|^{2}}{2}\rho_t 
\end{split}
\end{equation*}
where we used the inequality
$$a\cdot (b-a)\leq \frac12|b|^2-\frac12|a|^2.$$
Thus
\begin{equation*}
    \begin{split}
    \partial_t \mathcal{M}_p(\rho_t) &\leq - \frac{1}{2}\int |x|^{p-2}\left(\fint |x\cdot \vartheta|^2\right)\rho_t+\frac{1}{2}\fint\int |x|^{p-2}|T^\vartheta_t(x \cdot  \vartheta)|^2 \rho_t \\
    &\leq-\int |x|^{p-2}\frac{|x|^2}{2d}\rho_t +\frac{1}{2}\fint \left[\left(\int |x|^{p}\rho_t\right)^{\frac{p-2}{p}}\left(\int |T^\vartheta_t(x\cdot \vartheta)|^{p}\rho_t\right)^{\frac{2}{p}}\right] 
    \end{split}
\end{equation*}
using the equality
\begin{equation*}
    \fint_{\mathbb{S}^{d-1}} |x\cdot \vartheta|^{2}d\mathscr{H}^{d-1}(\vartheta)= \frac{|x|^{2}}{d}.
\end{equation*}
This is a particular case of a more general family of equalities, 
\begin{equation*}
    \fint |x \cdot \vartheta|^{p} = c_{p,d}|x|^{p},
\end{equation*}
which are true for certain constants $c_{p,d}<1$ (and indeed we have $c_{2,d}=1/d$).
Hence, finally, using Jensen's inequality, we get
 \begin{equation}
 \label{stime_finali_pmomenti}
     \begin{split}
        \partial_t \mathcal{M}_p(\rho_t)&\leq  -\frac{1}{2d}\mathcal{M}_{p}(\rho_t)+\frac{1}{2}\mathcal{M}_{p}(\rho_t)^{1-\frac{2}{p}}\left(\int \fint |y\cdot  \vartheta|^{p}\nu\right)^{\frac{2}{p}}\\
        &\leq  -\frac{1}{2d}\mathcal{M}_{p}(\rho_t)+\frac{1}{2}\mathcal{M}_{p}(\rho_t)^{1-\frac{2}{p}}c_{p,d}^{\frac{2}{p}}\mathcal{M}_{p}(\nu)^{\frac{2}{p}}.
     \end{split}
 \end{equation}
We deduce 
\begin{equation}
\label{Estimate_2_Moments}
\mathcal{M}_{p}(\rho_t) \leq \max \left\{\mathcal{M}_{p}(\rho_{0}), d^{p/2}c_{p,d}\mathcal{M}_{p}(\nu)\right\},
\end{equation}
since \eqref{stime_finali_pmomenti} gives a differential inequality of the form 
\[x' \leq -\frac{1}{2d}x +x^{1-2/p}\frac{C^{2/p}}{2},\]
and the left hand side is negative for 
\[x \geq Cd^{p/2}.\]
By Theorem \ref{Compactness_Wasserstein} we deduce that, if the target and starting measure have finite $p$-moments, the flow describes a compact curve in $W_q$ for any $q < p$.
Moreover, denoting by 
\[R(\mu)\coloneqq \inf\left\{r>0 \colon \text{ supp}(\mu) \subset \{|x|\leq r\}\right\}=\lim_{p \to \infty} \mathcal{M}^{1/p}_{p}(\mu),\] 
we can see that we have
\begin{equation*}
\begin{split}
    R(\rho_t)=\lim_{p \to \infty} \mathcal{M}_{p}(\rho_t)^{1/p}&\leq \max \left\{\lim_{p \to \infty} \mathcal{M}_{p}(\rho_{0})^{1/p}, d^{1/2} \lim_{p \to \infty} c_{p,d}^{1/p}\mathcal{M}_{p}(\nu)^{1/p}\right\}\\
    &=\max \left\{R(\rho_0),\sqrt{d}R(\nu)\right\},
\end{split}
\end{equation*}
since 
\[\lim_{p \to \infty} c_{d,p}= \lim_{p \to \infty} \lVert\vartheta\rVert_{L^{p}\left(\mathbb{S}^{d-1}\right)}=\lVert \vartheta \rVert_{L^{\infty}\left(\mathbb{S}^{d-1}\right)}=1.\]
Obviously this estimate is useless when the initial and target measures are not compactly supported. This estimate is not sharp, and one can show that the constant $\sqrt{d}$ can be improved by finer arguments on the direction of the velocity field $v$ on large balls far from the support of the target measure.

\subsection{Gaussian-target case}
We prove here that when the target measure $\nu$ is the standard Gaussian distribution, then $\lim_{t \to \infty}\rho_t$ in the Wasserstein topology exists and coincides with the Gaussian distribution $\nu$ itself. The argument we present can be adapted to isotropic Gaussian measures of arbitrary variance, but unfortunately does not generalize to other Gaussian distribution (whose covariance matrix is not a scalar multiple of the identity). 

Let the entropy function $\mathcal{E} \colon \mathscr{P}(\mathbb{R}) \to \mathbb{R}$ be defined by
    \begin{equation*}
    \mathcal{E}(\rho)=
    \begin{cases}
        \int_{\mathbb{R}^d}\rho \log \rho &\text{ if } \rho \ll \mathscr{L}^d,\\
        +\infty & \text{otherwise}.
    \end{cases}
\end{equation*}
\begin{theorem}
\label{Gaussian theorem}
    Let $\nu$ be the standard Gaussian distribution on $\mathbb{R}^d$. If $\rho_0$ has finite entropy and finite second moment, then the flow $(\rho_t)_{t\geq 0}$ is well posed and we have
    \[SW_2^2(\rho_t,\nu) \leq \frac{C}{t}\]
    where $C=C(\rho_0)=2(\mathcal E(\rho_0)+\mathcal M_2(\rho_0))$.
\end{theorem}
To prove this theorem we need the following preliminary result.
\begin{lemma}
\label{Lemma_ENTROPIA}
Given two absolutely continuous probability measures $\mu$, $\nu \in \mathscr{P}_2(\mathbb{R})$, consider the optimal transport map $T$ from $\mu$ to $\nu$.
Then we have
\[\int_{\mathbb{R}} \partial_y \mu (y) \left(T(y)-y\right) dy\leq \mathcal E(\nu)-\mathcal E(\mu). \]
\end{lemma}
\begin{proof}
    Consider the curve of measures $\omega(t)=((1-t)id+tT)_\#\mu$. Such a curve is a constant-speed geodesic (with respect to the $ W_{2}$ distance) connecting $\mu=\omega(0)$ to $\nu=\omega(1)$. 
    
    Since $\omega$ is a geodesic, we know it must solve the the continuity equation 
\[\partial_{t}\omega +\partial_{y}(\textbf{v}_{t}\omega)=0,\] 
where the velocity field is $\textbf{v}(t,y)= T(S_{t}^{-1}(y))-S_{t}^{-1}(y)$, with $S_{t}(y)\coloneqq (1-t)y+tT(y)$. We thus have
\begin{equation*}
    \begin{split}
    &\partial_{t}\int_{\mathbb{R}}\omega_t \log{\omega_t} dr = \int_{\mathbb{R}} \partial_{t}\omega_t\log{\omega_t}dr\\ &=- \int_{\mathbb{R}}\partial_{r}(\textbf{v}_{t}\omega_t)\log{\omega_t} dr=\int_{\mathbb{R}}\textbf{v}_{t}\partial_{y}\omega_t\\
    &=\int_{\mathbb{R}} \left[T\left(S_{t}^{-1}(y)\right)-S_{t}^{-1}(y)\right]\partial_{y}\omega_t dy,\\
    \end{split}
\end{equation*}
and valuating the expression at $t=0$, we obtain:
\begin{equation*}
    \partial_t\bigg|_{t=0}\mathcal{E}(\omega(t))=\int_{\mathbb{R}}\partial_y \mu(y)\left(T(y)-y\right)dy.
\end{equation*}
We now use the convexity of the function $t \mapsto \mathcal{E}(\omega(t))$, where $\omega$ is a Wasserstein geodesic in $\mathscr{P}(\mathbb{R}^d)$. Indeed, the entropy is convex along these geodesics (see \cite{McC} and \cite[Section 7.3.2]{santambrogio2015optimal}), and for any convex function $f$ one has $f'(0) \leq f(1)-f(0)$. This proves the claim.
\end{proof}
We are now ready to prove theorem \ref{Gaussian theorem}
\begin{proof}
Since $(\rho_t)_{t \geq 0}$ solves the continuity equation \eqref{CE}, we have
\begin{equation}
\begin{split}
\label{uno}
        &\partial_t \int \rho_t\log{\rho_t} dx= \int [\partial_t \rho_t \log{\rho_t} + \partial_t \rho_t ]dx=-\int [\nabla\cdot (v_t\rho_t) \log{\rho_t} - \nabla(v_t\rho_t)] dx\\
        &=-\int \nabla\cdot(\rho_t v_t)\log{\rho_t} dx= \int \rho_t v_t \cdot\frac{1}{\rho_t}\nabla\rho_t dx= \int v_t \cdot \nabla\rho_t dx\\
        &=\fint \left[\int \nabla \rho_t \cdot\vartheta \left(T^\vartheta_{t}(x \cdot \vartheta)-x \cdot \vartheta\right) dx\right] d\mathscr{H}^{d-1}(\vartheta)\\
        &=\fint \int_{\mathbb{R}} \partial_y \rho^\vartheta_{t}(y)\left(T^{\vartheta}_{t}(y)-y\right)dy d\mathscr{H}^{d-1}(\vartheta) \leq \fint \left[ \mathcal{E}\left(\nu^{\vartheta}\right)-\mathcal{E}\left(\rho^{\vartheta}\right)\right]d\mathscr{H}^{d-1}(\vartheta).
\end{split}
\end{equation} 
The last inequality comes from the inequality in Lemma \ref{Lemma_ENTROPIA}. We then obtain 
\begin{equation}
\label{Entropy_ineq}
    \partial_t \mathcal{E}(\rho_t)\leq \fint \mathcal{E}(\nu^\vartheta)-\fint \mathcal{E}(\rho^\vartheta).
\end{equation}

Now notice that the Gaussian measure minimizes the functional
\[\rho \mapsto \mathcal{E}(\rho)+\int \frac{|x|^2}{2}\rho.\]
\noindent Indeed $f(s)=s\log{s}$ is a convex function, and thus $f(t)\geq f(s)+f'(s)(t-s)=f(s)+(\log{s}+1)(t-s)$. Thus, for every absolutely continuous $\rho$, and if $\nu$ is the Gaussian distribution, 
\begin{dmath*}
    \int \rho \log{\rho} \geq \int \nu\log{\nu}+\int \log{\nu}(\rho-\nu)+\int (\rho-\nu)=\int \nu \log{\nu}-\int \frac{|x|^2}{2}\rho+\int \frac{|x|^2}{2}\nu,
\end{dmath*}
and therefore 
\[\int \left[\rho \log{\rho}+\frac{|x|^2}{2}\rho\right]\geq \int \left[ \nu \log{\nu}+\frac{|x|^2}{2}\nu \right] =0\]
Summing Equations \eqref{due} and \eqref{Entropy_ineq}, and using the fact that, for any $\vartheta \in \mathbb{S}^{d-1}$, $\nu^{\vartheta}$ is still a (one-dimensional) standard Gaussian, we obtain
\begin{equation*}
\begin{split}
    &\partial_t \left(\mathcal{E}(\rho_{t})+\int \frac{|x|^2}{2}\rho_{t} \right)\\
    &\leq \fint \left(\mathcal{E}\left(\nu^{\vartheta}\right)-\mathcal{E}\left(\rho_{t}^{\vartheta}\right)\right) + \fint \int \frac{|y|^2}{2}\nu^{\vartheta} - \fint\int \frac{|y|^2}{2}\rho_{t}^\vartheta-\frac{1}{2} SW_{2}^{2}(\rho_{t},\nu)\\
    &=\fint \left[ \left( \mathcal{E}\left(\nu^{\vartheta}\right)+\int\frac{|y|^2}{2}\nu^{\vartheta}\right)- \left( \mathcal{E}\left(\rho_{t}^{\vartheta}\right)+\int\frac{|y|^2}{2}\rho_{t}^\vartheta\right)\right]  - \frac{1}{2} SW^{2}_{2}(\rho_{t},\nu)\\
    &\leq -\frac{1}{2} SW^2_2(\rho_{t},\nu).
\end{split}
\end{equation*}
\noindent Hence, integrating with respect to time, and recalling that $\frac{1}{2}SW_2^2$ is a decrasing quantity along its gradient flow,
\begin{equation*}
    \begin{split}
        \frac{T}{2}SW_2^2(\rho_T,\nu)\leq \frac{1}{2}\int_{0}^{T}  SW^2_2(\rho_{t},\nu)dt &\leq -\left[\mathcal{E}(\rho_{T})-\mathcal{E}(\rho_{0})\right]-\int \left[\frac{|x|^2}{2}\rho_{T} - \frac{|x|^2}{2}\rho_{0} \right]dt\\  
        &\leq \mathcal{E}(\rho_{0})+\int \frac{|x|^2}{2}\rho_{0} \eqqcolon C
    \end{split}
\end{equation*}
for every $T \geq 0$. Thus
\[ SW^{2}_{2}(\rho_{t},\nu) \leq 2\frac{\mathcal{E}(\rho_0)+\int \frac{|x|^2}{2}\rho_{0}}{t}.\]

\end{proof}

The above proof is inspired by the proof of convergence of the IDT procedure in \cite{pitie2007automated}, presented as well in \cite[Theorem 5.2.2]{bonnotte2013unidimensional}.

\begin{Corollary}
Take $q\geq 2$. Then, if $\rho_0\in \mathscr P_q(\mathbb R^d)$ has finite entropy, we have
\begin{equation*}
    \rho_{t} \xrightarrow[t \to \infty]{W_{p}} \nu,
\end{equation*}
for all values of $p<q$.
\end{Corollary}
\begin{proof}
The proof is easy once we observe that \eqref{Estimate_2_Moments} and Theorem \ref{Compactness_Wasserstein} provide compactness of the flow $(\rho_{t})_{t \geq 0}$ in the spaces $(\mathscr{P}_p(\mathbb{R}^d), W_{2})$ provided $\rho_0$ has finite moments. This turns the convergence in the distance $SW_2$ into a more standard $W_p$ convergence.
for any $q<p$.
\end{proof}

\subsection{The flow may not converge to the target}
We can provide examples in which the flow does not converge to the target.
\begin{example}
Consider the following construction in $\mathbb{R}^2$: let the starting measure be
\[\rho_0=\frac{\delta_{(-1,0)}}{2}+\frac{\delta_{(1,0)}}{2}\]
and the target measure
\[\nu=\frac{\delta_{(0,a)}}{2}+\frac{\delta_{(0,-a)}}{2},\]
for a certain $a>0$ to be defined. We can show that there exists $a>0$ such that the configuration $\rho_0$ is stationary, \textit{i.e.} the velocity field associated with the SWF at time zero is $v_0=0$ and therefore the gradient flow is $\rho_t =\rho_0$ for any $t \geq 0$, so that the flow does not converge to the target. 
Calling $T^{\vartheta}$ the optimal transport map from $\rho_\vartheta$ and $\nu_\vartheta$, we have
\begin{equation*}
\begin{split}
    v_0 &= \fint \left(T^{\vartheta}(x \cdot \vartheta)-x\cdot \vartheta\right)\vartheta\\
    &=2\fint_{0}^{\frac{\pi}{2}}\left(a\sin\vartheta - \cos\vartheta \right)\begin{pmatrix}
    \cos\vartheta\\
    \sin\vartheta
    \end{pmatrix}-2 \fint_{-\frac{\pi}{2}}^{0}\left(\cos\vartheta + a\sin\vartheta\right) \begin{pmatrix}
    \cos\vartheta\\
    \sin\vartheta
    \end{pmatrix}\\
    &=2\fint_{0}^{\frac{\pi}{2}}\begin{pmatrix}
    a\sin\vartheta\cos\vartheta - \cos^2\vartheta\\ a\sin^2\vartheta - \cos\vartheta\sin\vartheta\end{pmatrix}-2\fint_{\frac{-\pi}{2}}^{0} \begin{pmatrix}
        \cos^2\vartheta + a\cos\vartheta\sin\vartheta\\
        \cos\vartheta \sin\vartheta + a\sin^2\vartheta
    \end{pmatrix}\\
    &=\begin{pmatrix}
    a\fint_{0}^{2\pi}\frac{\sin(2\vartheta)}{2}\text{sgn}(\vartheta)-\fint_{0}^{2\pi}\cos^2\vartheta\\
    0
    \end{pmatrix}.
\end{split}
\end{equation*}
All we have to do now is imposing the first row of the above matrix equal to zero, obtaining:
\[a=\frac{\pi}{2}.\]
\end{example}
Note that the above example is not new at all and similar computations on the case where both the target and the starting measure are composed of two symmetric atoms are also illustrated in \cite[Section 1.6]{TanFlaDel} or \cite[Figure 2]{BonRabPeyPfi}. 
\begin{example}
    Consider the following construction in $\mathbb{R}^2$: let the starting measure be 
    \[\rho_0=\frac{\mathscr{H}^{1}\left((a,b)\times\{0\}\right)}{|b-a|}\]
    and the target measure $\nu$ be any radial distribution. Then, by the symmetry in the construction, we can see that we always have $\mathrm{supp}(\rho_t)\subset \mathbb R\times\{0\}$. Thus, in general, this flow will not converge to the target measure.
\end{example}
Notice that in both these examples we cannot assume $\rho_0$ to be absolutely continuous with respect to the Lebesgue measure. This is indeed coherent with the strategy used to prove the convergence of the Gaussian-target case, which relied on entropy methods.

\section{The flow map of the SWF does not provide optimal transport}
\label{S:4}

This section contains a counterexample to Conjecture 2 presented in the introduction. We show the following:
\begin{center}
\textit{It is not true that, for any initial distribution $\rho_0$ and any target distribution $\nu$, the sliced-Wasserstein flow converges to the target measure $ \nu$ itself and the limit of the flow map arising from the Lagrangian description of the model exists and is the optimal transport map from  $\rho_0$ to $\nu$.}
\end{center}
Using the idea developed in \cite{lavenant2022flow}, we can write a necessary condition which must hold if we want Conjecture 2 to be true: take $\rho_0$ sufficiently smooth and quickly decaying at infinity, and assume that Conjecture 2 holds not only for $\rho_0$ but for all $(\rho_t)_{t\geq 0}$. This means that the flow $Y_t$ is well defined for any $t\geq 0$, that the limit $\lim_{t \to \infty}Y_t\eqqcolon T$ exists, and that it is the optimal transport map between $\rho_0$ and $\nu$. Under these assumptions, we have that for any $t \geq 0$ the SWF provides optimal transport from $\rho_t$ to $\nu$. That is, the map $T\circ Y_{t}^{-1}$ is an optimal transport map between $\rho_t$ and $\nu$. Thus, $Y_t \circ T^{-1}$ is also an optimal transport map between $\nu$ and $\rho_t$. Let us denote by $S$ the map $T^{-1}$. Making use of theorem \eqref{Brenier's_thm}, we have that the following condition must hold
\begin{equation} 
\label{Necessary_Condition_1}
\forall t \geq 0, \forall x \in \mathbb{R}^d, \hspace{1mm} \text{ the Jacobian of } Y_t \circ S (x) \text{ is a symmetric matrix }.
\end{equation}
This Jacobian matrix reads $DY_t(S)DS$. Moreover, differentiating \eqref{ODE.} with respect to $x$, we see that
\[\frac{\partial DY_t}{\partial t}=-\left(\fint_{\mathbb{S}^{d-1}} \varphi''_{t,\vartheta}(x \cdot \vartheta)\vartheta \otimes \vartheta \hspace{3mm}d\vartheta\right),\]
together with $DY_{0}=\text{Id}$. Thus, differentiating the Jacobian of $Y_t \circ S$ with respect to time, we have that condition \eqref{Necessary_Condition_1} implies  
\[\forall t \geq 0, \forall x \in \mathbb{R}^{d}, \hspace{5mm} \left[\fint_{\mathbb{S}^{d-1}} \varphi''_{t,\vartheta}\left(S(x)\cdot \vartheta\right)\vartheta \otimes \vartheta\right]DS(x) \hspace{3mm} \text{ is a symmetric matrix.}\]
Evaluating this expression at $t=0$ and remembering that $DS(x)$ is symmetric for any $x$ (since our assumption is that $S$ is an optimal transport map between $\nu$ and $\rho_0$) we conclude that the matrices
\[\fint_{\mathbb{S}^{d-1}}\varphi''_{\vartheta}\left(S(x) \cdot \vartheta\right)\vartheta \otimes \vartheta \hspace{3mm}\text{ and } \hspace{3mm} DS(x) \hspace{3mm}\]should commute for any $x$.
Composing both the matrices with $S^{-1}=T$ on the right hand side and using the identity $DS(S^{-1})=[DT]^{-1}$, and the fact that a symmetric matrix $A$ commutes with an invertible matrix $B$ if and only if it commutes with $B^{-1}$, we conclude that:
\[\forall x \in \mathbb{R}^{d}, \hspace{5mm} \fint_{\mathbb{S}^{d-1}}\varphi''_{\vartheta}\left(x \cdot \vartheta\right)\vartheta \otimes \vartheta \hspace{3mm} \text{and  } \hspace{3mm} DT(x) \hspace{3mm}\text{commute}.\]
\\
Using again Theorem \ref{Brenier's_thm}, namely the fact that $T=Du$ for a convex map $u\colon \mathbb{R}^{d} \to \mathbb{R}$, we can write
\[\forall x \in \mathbb{R}^{d}, \hspace{5mm} \fint_{\mathbb{S}^{d-1}}\varphi''_{\vartheta}\left(x \cdot \vartheta\right)\vartheta \otimes \vartheta \hspace{3mm} \text{and  } \hspace{3mm} D^2u(x) \hspace{3mm}\text{ commute.}\]

For the sake of simplicity we will call $\rho_0=\rho$ in the following. 
Denote also by $u_\vartheta \colon \mathbb{R} \to \mathbb{R}$ the map whose derivative is the transport map $T^\vartheta$ between $\rho^\vartheta$ and $\nu^\vartheta$. Writing the Monge-Ampére equation for $\rho^{\vartheta}$ and $\nu^{\vartheta}$ we have
\begin{equation*}
    u''_{\vartheta}(z)=\frac{\rho^{\vartheta}(z)}{\nu^{\vartheta}\left(T^{\vartheta}(z)\right)} \hspace{3mm} \forall z \in \mathbb{R}.
\end{equation*}
Since  we have the equality $\varphi''_{\vartheta}(z)=1-u''_{\vartheta}(z)$, we get
\begin{equation*}
M_{i,j}(x)\coloneqq \fint \left( 1 -\frac{\rho^{\vartheta}(x \cdot \vartheta)}{\nu^{\vartheta}\left(T^{\vartheta}(x \cdot \vartheta)\right)}\right) \vartheta_{i}\vartheta_{j}= \frac{\delta_{ij}}{d} - \fint \frac{\rho^{\vartheta}(x \cdot \vartheta)}{\nu^{\vartheta}(T^{\vartheta}(x \cdot \vartheta))}\vartheta_{i}\vartheta_{j}.
\end{equation*}
If we choose $\rho$ and $\nu$ to be even measures ($\rho(x)=\rho(-x)$, $\nu(x)=\nu(-x)$ for all $x \in \mathbb{R}^d$), we have $T^{\vartheta}(0)=0$ for all $\vartheta \in \mathbb{S}^{d-1}$. Moreover, we also choose $\nu$ to be radial (which is stronger than even), so that the value $\nu^\vartheta(0)$ does not depend on $\vartheta$. Our goal can be rewritten as follows: we want to show that, in general, the following two matrices do not commute:
\begin{equation}
\label{Condition_B}
    \frac{\text{Id}}{d}-M(0)=\fint \frac{\rho^{\vartheta}(0)}{\nu^{\vartheta}(0)} (\vartheta \otimes \vartheta) d\vartheta =c\fint \rho^{\vartheta}(0)(\vartheta \otimes \vartheta) d\vartheta \hspace{3mm} \text{ and } \hspace{3mm} D^2u(0).
\end{equation}

Since we are free to choose any even initial measure $\rho$, let us impose $\rho =(\nabla u)^{-1}_{\#}\nu$, for $u(x)=|x|^{2}/2+\varepsilon\varphi(x)$, where $\varphi$ is a smooth, compactly supported and even function. It is well known that, for $\varepsilon$ small enough, $u$ is convex and $\nabla u$ is a $C^{\infty}$ diffeomorphism that coincides with the identity outside of a compact set which moreover, by Brenier's theorem, is the optimal transport map from $\rho$ to $\nu$.
Finally the Monge-Ampére equation guarantees that, if $\varphi$ is even, so is $\rho$.  
We seek now for the expression of $\rho^{\vartheta}(0)$ in terms of $\varphi$. Again by the Monge-Ampére equation for the transport problem from $\rho$ to $\nu$, we deduce that, for any $x \in \mathbb{R}^d$, we have
\[\rho(x)=\det D^2u(x)\nu(\nabla u(x)).\]
In our case $\nabla u (x) = x + \varepsilon\nabla\varphi(x)$ and $\det D^{2}u(x)= 1+\varepsilon\Delta\varphi(x)+ \mathcal{O}(\varepsilon^{2})$.
We deduce that for any $x \in \mathbb{R}^d$
\begin{equation*}
\begin{split}
\rho(x)=&\left(1+\varepsilon\Delta\varphi+\mathcal{O}(\varepsilon^2)\right)\left(\nu (x) + \varepsilon D\nu(x) \cdot \nabla \varphi(x)\right) +\mathcal{O}(\varepsilon^2)\\
=&\nu(x) + \varepsilon\left(D\nu(x) \cdot \nabla \varphi(x)+\Delta\varphi(x)\nu (x)\right)+\mathcal{O}(\varepsilon^2).
\end{split}
\end{equation*}
By the definition of $\rho^\vartheta$, we get
\begin{equation*}
\label{rhotheta}
\rho^{\vartheta}(0)=\int_{\vartheta^{\perp}}\rho(y)dy=\int_{\vartheta^{\perp}}\left(\nu(y) + \varepsilon\left(D\nu(y) \cdot \nabla \varphi(y)+\Delta\varphi(y)\nu (y)\right)\right)dy+\mathcal{O}(\varepsilon^{2}).
\end{equation*}

We now impose $\varphi(x)=\varepsilon \psi(x) + \eta(x)$, with $\psi$, $\eta$ even functions (so that $\varphi$ is even) and such that $\text{supp}(\psi) \subseteq B(0,1)$ and $\text{supp}(\eta) \subseteq B( R e_d, 1) \cup B( -R e_d, 1) $ for a fixed $R>>1$, $e_d$ being the last vector of the canonical basis in $\mathbb{R}^d$. In this way $D^2 u (0) = \mathrm{Id}+\varepsilon^2 D^2\psi (0)$ and so, by \eqref{Condition_B}, our aim becomes to prove that the matrix $\fint \rho^{\vartheta}(0) \vartheta \otimes \vartheta$ does not commute with $D^2\psi(0)$. 
We have 
\begin{equation*}
    \begin{split}
        \rho^{\vartheta}(0) = & \int_{\vartheta^{\perp}} \nu(y)  dy +\varepsilon\int_{\vartheta^{\perp}} \left(D\nu(y) \cdot \nabla \eta(y)+\Delta\eta(y)\nu (y)\right)dy + \mathcal{O}(\varepsilon^2)\\
        =& \int_{\vartheta^{\perp}} \nu(y)  dy  +\varepsilon\int_{\vartheta^{\perp}}\nu(y)\eta_{\vartheta \vartheta}(y)dy+ \mathcal{O}(\varepsilon^2),
    \end{split}
\end{equation*}
after integration by parts.
We deduce that our goal becomes now to prove the following fact: \say{$D^2 \psi(0)$ does not commute with $A$}, where
\begin{equation*}
\begin{split}
\label{Tensor_Description_Of_The_Matrix}
A=&\fint\int_{\vartheta^{\perp}} \nu(y)\eta_{\vartheta \vartheta}(y)(\vartheta \otimes \vartheta)  dy d\vartheta\\
    =& \fint \int_{\vartheta^{\perp}} \nu(y)(\vartheta \otimes \vartheta) D^2\eta(y)(\vartheta \otimes \vartheta) dy d\vartheta\\
    =& 2\fint \int_{\vartheta^{\perp}} \nu(y)(\vartheta \otimes \vartheta) D^2\Tilde{\eta}(y)(\vartheta \otimes \vartheta) dy d\vartheta,
\end{split}
\end{equation*}
and $\Tilde{\eta}$ is the restriction of $\eta$ to $B(Re_d, 1)$. For convenience, with a slight abuse of notation, in the following we will call $\eta$ this very function. Notice in particular that $\eta(x)=0$ if $x_d \leq 0$.
Since $\psi$ is any even function, in order to get our claim we just have to ensure that the matrix $A$ is not a multiple of the identity matrix. To do so, we can for example check that the diagonal elements of $A$ are not equal.

In the remaining part of this counterexample, for simplicity, we restrict to the case $d=2$. In this case, we can just check that we have $A_{11}-A_{22} \neq 0$: take $\eta(y_1,y_2)=\eta^r(y_1,y_2)=a^r(y_1)b(y_2)$ for $a^r$ and $b$ positive functions, even on their support and of unitary integrals, supported respectively in $[-r,r]$, $[R,R+1/2]$ (so that, if $r$ is small, $\text{supp}(\eta^r) \subseteq B(R e_2,1)$). Observe that $a^r \rightharpoonup \delta_{0}$ when $r \to 0$.\\

Using the notation $\eta^r_{ij} = (D^2\eta^r)_{ij}$, we have 
\[A_{11}=2\fint \int_{\vartheta^{\perp}}\vartheta_1^2(\vartheta_{1}^2\eta^r_{11}+2\vartheta_{1}\vartheta_{2}\eta^r_{12}+\vartheta_2^2\eta^r_{22})\nu(y)dy d\vartheta,\]
\[A_{22}=2\fint \int_{\vartheta^{\perp}}\vartheta_2^2(\vartheta_{1}^2\eta^r_{11}+2\vartheta_{1}\vartheta_{2}\eta^r_{12}+\vartheta_2^2\eta^r_{22})\nu(y)dy d\vartheta.\]
Since in dimension $d=2$, for any $y \in \vartheta^{\perp}$ one has 
\begin{equation*}
\label{Dimension_2_Trick}(\vartheta_1,\vartheta_2)=\left(-\frac{y_2}{|y|},\frac{y_1}{|y|}\right),
\end{equation*}
we have
\begin{equation}
\begin{split}
\label{A_11_And_A_22}
A_{11}=&2\fint \int_{\vartheta^{\perp}}\frac{y_2^2}{|y|^{4}}(y_2^2\eta^r_{11}-2y_{1}y_{2}\eta^r_{12}+y_1^2\eta^r_{22})\nu(y)dy d\vartheta,\\
A_{22}=&2\fint \int_{\vartheta^{\perp}}\frac{y_1^2}{|y|^4}(y_2^2\eta^r_{11}-2y_{1}y_{2}\eta^r_{12}+y_1^2\eta^r_{22})\nu(y)dy d\vartheta.
\end{split}
\end{equation}
To conclude we need the following 
\begin{lemma}
\label{Integration_Lemma}
    Given $f\colon \mathbb{R}^d \to \mathbb{R}$ such that the following integrals are well defined, we have
    \[\fint_{\mathbb{S}^{d-1}}\int_{\vartheta^{\perp}}f(y)dyd\vartheta=\frac{(d-1)\omega_{d-1}}{d\omega_{d}}\int_{\mathbb{R}^{d}}\frac{f(x)}{|x|}dx\]
    where $\omega_k$ stands for the measure of the unit ball in dimension $k$.
\end{lemma}
\begin{proof}
First restricting to compactly supported functions $f$ and then removing this restriction, it is clear, by the Riesz representation theorem for the dual of $C^0$, that there exists a locally finite measure $\mu$ on $\mathbb R^d$ such that 
 $\fint_{\mathbb{S}^{d-1}}\int_{\vartheta^{\perp}}f(y)dyd\vartheta=\int f d\mu$. It is also clear by symmetry reason that $\mu$ is a radial measure. Hence, in order to identify $\mu$ it is enough to compute the mass that it gives to every ball $ B(0,R)$. Taking for $f$ the indicator function of such a ball we obtain
 $\mu(B(0,R))=\omega_{d-1}R^{d-1}$. This is the same as taking for $\mu$ the measure with density $\frac{(d-1)\omega_{d-1}}{d\omega_{d}}\frac{1}{|x|}$, which concludes the proof.
\end{proof}
Applying Lemma \ref{Integration_Lemma} to Equations \eqref{A_11_And_A_22}, we deduce
\begin{equation*}
    \fint_{\mathbb{S}^{1}} \int_{\vartheta^{\perp}} f(y) dy d\vartheta = \frac{1}{\pi}\int_{\mathbb{R}^{d}} \frac{f(x)}{|x|}dx.
\end{equation*}
Using this formula, the difference between the diagonal elements of the $2 \times 2$ - matrix $A$, up to multiplication by the constant $\pi$, reads 
\begin{equation*}
    \begin{split}
        A_{11}-A_{22}=& \int_{\mathbb{R}^2} \left(\frac{x_2^2-x_1^2}{|x|^5}(x_2^2\eta^r_{11}-2x_1x_2\eta^r_{12}+x_1^2\eta^r_{22})\right)\nu(x)dx_1dx_2.\\
    \end{split}
\end{equation*}
Since 
\[h_{i}(x)\coloneqq \frac{x_i^2}{|x|^5}\nu(x)\]
can be assumed smooth enough in $\mathbb{R}^2\setminus \{0\}$, when $r \to 0$ the integrals defining $A_{11}$ and $A_{22}$ stay bounded (remembering that $\eta^r$ is compactly supported). 
Setting 
\[f(x_1,x_2)\coloneqq h_{2}(x)-h_1(x)= \frac{x_2^2-x_1^2}{|x|^5}\nu(x_1,x_2),\]
we have 
\begin{equation*}
    \begin{split}
        A_{11}-A_{22}=& \int_{\mathbb{R}^2}f(x)(x_2^2 \partial_{11}a^rb-2x_1x_2 \partial_{1}a^r\partial_{2}b+x_1^2 a^r \partial_{22}b)dx_1 dx_2.
    \end{split}
\end{equation*}
We need to be sure that this quantity is not zero as soon as $r \to 0$. Integrating by parts with respect to the variable $x_1$, we obtain
 \begin{equation*}
     \begin{split}
         A_{11}-A_{22} = & \int \left( x_2^2 b \int \partial_{11}f a^{r}+2x_2 \partial_2 b \int (\partial_1 f x_1 + f) a^{r} + \partial_{22} b \int f x_1^2 a^{r} \right) dx_2.
     \end{split}
 \end{equation*}
 Sending $ r \to 0$, this gives
 \begin{equation*}
     \begin{split}
         A_{11}-A_{22}=\int x_2^2 b \partial_{11}f(0,x_2)+2x_2\partial_2 b f(0,x_2) dx_2.
     \end{split}
 \end{equation*}
 Integrating by parts with respect to $x_2$,
 \begin{equation*}
     \begin{split}
         A_{11}-A_{22} = & \int b\left(\partial_{11}f(0,x_2) x_2^2-2f(0,x_2)-2x_{2}\partial_{2}f(0,x_2)\right) dx_2. 
     \end{split}
 \end{equation*}
 Remembering that $x_2$ can be assumed strictly positive (due to the fact that $\eta^r$ is symmetric and supported away from zero), we have
 \begin{dmath*}
     \partial_1 f=\nu x_1 \frac{-2|x|^2-5x_2^2+5x_1^2}{|x|^7}+\frac{x_2^2-x_1^2}{|x|^5}\nu_1, 
 \end{dmath*}
 then
 \begin{dmath*}
     f_{11}(0,x_2)=-7\frac{\nu}{|x|^5}+\frac{\nu_{11}}{|x|^3}
 \end{dmath*}
 and, by similar computations,
 \begin{dmath*}
     f_2(0,x_2)=-3\frac{\nu}{|x|^4}+\frac{\nu_2}{|x|^3}.
 \end{dmath*}
Therefore
\begin{dmath*}
    A_{11}-A_{22}=\int b \left(-3\frac{\nu}{x_2^3}+\frac{\nu_{11}}{x_2}-2\frac{\nu_2}{x_2^2}\right)dx_2.
\end{dmath*}
Considering that $\nu$ has been chosen to be radial, we make use of the relations
\begin{dmath*}
    \label{Derivate_Radial}
    \nu_{i}=\nu'\frac{x_i}{|x|}
\end{dmath*}
(where $\nu'$ -- and later $\nu''$ -- stands for the derivative of $\nu$ in the radial direction), and 
\begin{dmath*}
    \label{Hessian_Radial}
    \nu_{ij}=\frac{\nu'}{|x|}\delta_{ij}+\left(\frac{\nu''}{|x|^2}-\frac{\nu'}{|x|^3}\right)x_ix_j,
\end{dmath*}
to get that $A_{11}-A_{22}=0$ for any choice of $b$ if and only if
\begin{dmath*}
    -3\frac{\nu}{|x|^3}-\frac{\nu'}{|x|^2}=0\\
\end{dmath*}
holds for any $x \in \{0\}\times \mathbb{R}^{+}$: indeed, since $b$ is free to be chosen, so is its support (which we defined to be of the form $[R,R+1/2]$). 
Still remembering that $x_2$ can be assumed strictly positive, we get that this is equivalent to
\begin{dmath}
\label{ODE_fin}
        3\nu + |x|\nu'=0.
\end{dmath}
To conclude the counterexample it is sufficient to choose a radial target measure $\nu$ not satisfying \eqref{ODE_fin}. By the way, no probability density can solve \eqref{ODE_fin} since a solution $\nu$ should satisfy $\nu \propto |x|^{-3}$, which implies that $\nu$ is not a probability distribution, as it is not integrable.
\bigskip




{\bf Acknowledgments} The authors acknowledge the support of the Lagrange Mathematics and Computation Research Center, where the project started during a Master internship of the first author. The support of the European union via the ERC AdG 101054420 EYAWKAJKOS is also acknowledged.

\printbibliography

\end{document}